\documentclass[12pt,a4paper]{amsart}
\usepackage[abbrev]{amsrefs}
\usepackage{amssymb}
\usepackage{color}

\theoremstyle{plain}
  \newtheorem{theorem}{Theorem}[section]
  
  \newtheorem{corollary}[theorem]{Corollary}
  \newtheorem{proposition}[theorem]{Proposition}

\theoremstyle{definition}
  \newtheorem{definition}[theorem]{Definition}
  
\theoremstyle{remark}
  \newtheorem{remark}[theorem]{Remark}
  \newtheorem*{ack}{Acknowledgment}

\numberwithin{equation}{section}

\def\supp{\mathop{\mathrm{supp}}\nolimits}
\begin{document}
\title[]{Invariant metric under deformed Markov  embeddings\\ with overlapped supports}
\author[]{Hiroshi Matsuzoe $^{\flat}$}
\author[]{Asuka Takatsu$^{\dagger}$}
\date{\today}
%
\thanks{$\flat$
Department of Computer Science, Nagoya Institute of Technology, Nagoya 466-8555, Japan
({\sf matsuzoe@nitech.ac.jp})}
\thanks{$\dagger$
Department of Mathematical Sciences, Tokyo Metropolitan University, Tokyo 192-0397, Japan
({\sf asuka@tmu.ac.jp})/
Mathematical Institute, Tohoku University, Sendai 980-8578, Japan/
RIKEN Center for Advanced Intelligence Project (AIP), Tokyo 103-0027, Japan}
\maketitle
\begin{abstract}
Due to \v{C}encov's theorem, there exists a unique family of invariant symmetric $(0,2)$-tensor fields on the space of positive probability measures on a set of $n$-points indexed by $n\in \mathbb{N}$ under Markov embeddings.
We  deform  Markov embeddings keeping sufficiency,  and prove existence and uniqueness of  invariant families under the embeddings.
\end{abstract}

\section{Introduction}
In the present paper, we investigate geometry  on  finite sample spaces.
For $n\in \mathbb{N}$,  set
\[
\Omega_{n+1}:=\{1, \cdots, n+1\},\quad
\mathcal{P}_{n}:=\left\{ p:\Omega_{n+1} \to (0,1) \ \bigg|\ \sum_{i\in \Omega_{n+1}} p(i)=1\right\}.
\]
The set $\mathcal{P}_{n}$ is regarded as a submanifold of $\mathbb{R}^{n+1}$ 
via an embedding $\iota_n: \mathcal{P}_{n} \to \mathbb{R}^{n+1}$ defined by $\iota_n(p)=(p(i))_{i=1}^{n+1}$.
From the viewpoint of statistics,
a Riemannian metric $g_n$ on $\mathcal{P}_n$ is required to be  invariant under the change of reference measures on the sample space $\Omega_{n+1}$.   
This invariance is extended to the invariance of a family $\{g_n\}_{n\in \mathbb{N}}$ of Riemannian metrics  under Markov embeddings.
\begin{definition}
\label{ME}
Let  $n, N\in \mathbb{N}$ with $n\leq N$.
\begin{enumerate}
\setlength{\leftskip}{-10pt}
\item
A family $\{Q_i\}_{i\in \Omega_{n+1}}$ of probability measures on $\Omega_{N+1}$ is a \emph{Markov partition} if 
\[
\bigsqcup_{i\in \Omega_{n+1}} \supp(Q_i)=\Omega_{N+1}
\]
holds, where $\supp(Q_i)$ stands for the  support of $Q_i$. 
\item
A map $F^N_n:\mathcal{P}_{n} \to \mathcal{P}_N$ is a \emph{Markov embedding}
if there exists a Markov partition $\{Q_i\}_{i\in \Omega_{n+1}}$ 
such that
\[
F^N_n(p):=\sum_{i\in \Omega_{n+1}} p(i) Q_i \quad \text{for any $p\in \mathcal{P}_{n}$.}
\]
\end{enumerate}
\end{definition}
Markov embeddings are closely related to  \emph{sufficient statistics}. 
\begin{proposition}{\rm(Fisher--Neyman characterization,\,\cite{AJLS-book}*{Theorem 5.3})}\label{ss}
Let $n,m\in \mathbb{N}$.
Fix a submanifold $\mathcal{P}$ of $\mathcal{P}_m$ smoothly parameterized by an open set  $M$ in a Euclidean space.
The diffeomorphism is denoted by  $p_{\ast}:M\to \mathcal{P}$. 
A  map  $\kappa : \Omega_{m+1} \to \Omega_{n+1}$  is {\em a Fisher--Neyman sufficient statistic} for~$\mathcal{P}$ if and only if  there exist 
$s:  \Omega_{n+1} \times M \to \mathbb{R}$ and  $t:\Omega_{m+1} \to \mathbb{R}$ such that 
\[
p_\xi (\omega)=s(\kappa (\omega),\xi)\cdot  t(\omega) \quad
\text{for any  $\xi\in M$ and $\omega \in \Omega_{m+1}$.}
\] 
\end{proposition}
See Remark~\ref{mke}~(1) for the relation between Markov embeddings and Fisher--Neyman sufficient statistics,
and \cite{AJLS-book}*{Section 5.1} for a detailed treatment of sufficient statistics.

\v{C}encov~\cite{Ce} showed that there exists a unique family of invariant Riemannian metrics  on $\mathcal{P}_n$ indexed by $n\in \mathbb{N}$ under Markov embeddings. 
This invariant Riemannian metric is called the \emph{Fisher metric}.
For each $i\in \Omega_{n+1}$,  define the function $\ell^i:\mathbb{R}^{n+1} \to \mathbb{R}$  by 
\[
\ell^i((x^j)_{j=1}^{n+1}):=\log |x^i|.
\]
\begin{definition}
The {\it Fisher metric} $g_n^F$ on $\mathcal{P}_{n}$ is a Riemannian metric  given by
\begin{align*}
g_n^F(X_p, Y_p)
&:=\sum_{i\in \Omega_{n+1}} p(i) X_p (\ell^i \circ \iota_n)  Y_p (\ell^i \circ \iota_n) \quad
\text{for any $p\in \mathcal{P}_{n}$ and $X,Y\in  \mathfrak{X}(\mathcal{P}_{n})$.}
\end{align*}
\end{definition}
\begin{theorem}{\rm (\cite{Ce}*{Theorem 11.1}, \cite{AJLS-book}*{Theorem 2.1})}\label{cencov}
For $n\in \mathbb{N}$, let $A_n$ be a continuous symmetric $(0,2)$-tensor field on $\mathcal{P}_{n}$.
Assume that for any $n, N\in \mathbb{N}$ with $n\leq N$, and   each Markov embedding $F^N_n:\mathcal{P}_{n} \to \mathcal{P}_N$,  
the pullback of $A_N$ by  $F^N_n$ is always $A_n$.
Then there exists $\lambda\in \mathbb{R}$ such that  $A_n=\lambda g_n^F$ for any $n\in \mathbb{N}$.
\end{theorem}

\v{C}encov's theorem has been extended  in various directions.
For example, this was extended to the cone over $\mathcal{P}_n$ (Remark \ref{Ca}, see also \cite{Ca})
and  to continuous sample spaces (see~\cite{AJLS, AJLS-book} and  the references therein).
However, as far as the authors know,
there is no attempt to deform Markov embeddings and to generalize \v{C}encov's theorem.

In this paper, we first  introduce two embeddings from $\mathcal{P}_n$ to $\mathcal{P}_{n+1}$ related to sufficient statistics
(Definition \ref{patch}).
We construct  a family of invariant symmetric $(0,2)$-tensor fields  on $\mathcal{P}_n$ under one kind of the  embeddings, 
where invariant families are not uniquely determined  (Proposition \ref{inv}).
We formulate a conformal condition  and  a parallel condition  to classify these invariant families
and  prove a uniqueness result (Theorem \ref{diag}).

\section{Deformed  Markov  embeddings}
We  deform  Markov embeddings so that the supports of the corresponding probability measures are overlapped while keeping sufficiency.
Let  $\mathfrak{S}_{n}$ be the symmetric group of degree~$n$.
\begin{definition}\label{patch}
Let $n\in \mathbb{N}$.
\begin{enumerate}
\setlength{\leftskip}{-10pt}
\item
A family $\{Q_i\}_{i\in \Omega_{n+1}}$ of probability measures on $\Omega_{n+2}$ is a \emph{Markov patch} 
if  there exist $\sigma \in \mathfrak{S}_{n+2}$ and $\{a_i\}_{i\in \Omega_{n+1}}$ such that $a_i\in (0,1)$ and
\[
Q_i:=a_i \delta_{\sigma (i)}+(1-a_i)\delta_{\sigma(n+2)}
\quad
\text{for any $i\in \Omega_{n+1}$,}
\]
where $\delta_{\ast}$ stands for Dirac's delta measure on $\Omega_{n+2}$.
\item
A map $G_n:\mathcal{P}_{n} \to \mathcal{P}_{n+1}$ is a \emph{patched Markov  embedding}
if there exists a Markov patch $\{Q_i\}_{i\in \Omega_{n+1}}$ 
such that
\[
G_n(p):=\sum_{i\in \Omega_{n+1}} p(i) Q_i
\quad
\text{for any $p\in \mathcal{P}_{n}$.}
\]
\end{enumerate}
In the case of  $a_i\equiv \alpha$,  we call  $\{Q_i\}_{i\in \Omega_{n+1}}$ and  $G_n:\mathcal{P}_{n} \to \mathcal{P}_{n+1}$ 
a \emph{Markov scalar  patch} and a \emph{scalar patched Markov  embedding}, respectively.
\end{definition}

\begin{remark}\label{mke}
We confirm  how  Markov embeddings  and patched Markov embeddings relate to  sufficient statistics.
\begin{enumerate}
\setlength{\leftskip}{-10pt}
\item
Fix a Markov embedding  $F_n^N: \mathcal{P}_{n} \to \mathcal{P}_N$ determined by a Markov partition $\{Q_i\}_{i\in \Omega_{n+1}}$.
Let $M \subset \mathbb{R}^{n}$ be  an open set being diffeomorphic to $\mathcal{P}_{n}$,
where the diffeomorphism is denoted by  $\rho_{\ast}:M\to \mathcal{P}_{n}$. 
Set $\mathcal{P}:=F^N_n(\mathcal{P}_{n})$  and 
$p:= F_n^N \circ \rho: M\to \mathcal{P}$.
Then a  map $\kappa :\Omega_{N+1} \to \Omega_{n+1}$  given by 
\[
\kappa(I):=i \quad \text{if $I\in \supp (Q_i)$}
\] 
is a Fisher--Neyman sufficient  statistic for $\mathcal{P}$, where we choose
  \[
s(i,\xi):=\rho_\xi(i), \quad t(I):=\sum_{i\in \Omega_{n+1} }Q_{i}(I). 
 \] 
\item
Fix a  patched Markov embedding  $G_n: \mathcal{P}_{n} \to \mathcal{P}_{n+1}$ determined by a Markov patch $\{Q_i\}_{i\in \Omega_{n+1}}$ of the form
\[
\quad\qquad Q_i=a_i \delta_{\sigma (i)}+(1-a_i)\delta_{\sigma(n+2)}\quad \text{for some $\sigma \in \mathfrak{S}_{n+2}$ and $\{a_i\}_{i\in \Omega_{n+1}}$ such that $a_i \in (0,1)$.}
\]
Fix $j\in \Omega_{n+1}, b\in (0,1)$  and set 
\[
a_{\min}:= \min\left\{a_i \, | \, i\in \Omega_{n+1}, i\neq j\right\}, \quad
a_{\max}:=\max\left\{a_i \, | \, i\in \Omega_{n+1}, i\neq j\right\}. 
\]
We take 
\[
c \in  \left( b a_j + (1-b)a_{\min}, \  b a_j + (1-b) a_{\max} \right) 
\]
if $a_{\min}<a_{\max}$, and otherwise  $c := ba_j+(1-b)a_{\min}$.
Then the subset of $\mathcal{P}_{n}$ defined by 
\[
\mathcal{P}':=\left\{p\in \mathcal{P}_{n} \ \Big|\ p(j)=b, \ \sum_{i\in \Omega_{n+1}} a_i p(i)=c \ \right\}
\]
is diffeomorphic to an  open set  $M \subset \mathbb{R}^{d}$, where
\[
d:=\begin{cases}
n-1 & \text{if  $a_{\min}=a_{\max}$},\\
n-2 & \text{otherwise.}
\end{cases}
\]
We denote by $\rho_{\ast}:M\to \mathcal{P}'$ the diffeomorphism.
We define  a map $\kappa:\Omega_{n+2} \to \Omega_{n+1}$  by 
\[
\kappa(I):=
\begin{cases}
i  & \text{if $I\in \supp(Q_i) \cap \sigma(\Omega_{n+1})$},  \\
j  & \text{if $I=\sigma(n+2)$}.
\end{cases}
\]
 Then  $\kappa:\Omega_{n+2} \to \Omega_{n+1}$
 is a Fisher--Neyman sufficient  statistic for $\mathcal{P}:=G_n(\mathcal{P}')$
 with a parametrization $p: = G_n \circ\rho:M \to \mathcal{P}$, where we choose
\[
\qquad
s(i, \xi):=\rho_\xi(i),
\quad
t(I):=\sum_{i\in \Omega_{n+1}} Q_i(I) \cdot \left( \sum_{i\in \Omega_{n+1}}  \delta_{\sigma(i)}(I) +\frac{1-c}{b \cdot \displaystyle\sum_{i\in \Omega_{n+1}} (1-a_i)}\delta_{\sigma(n+2)}(I) \right) .
\]
\end{enumerate}
\end{remark}

Throughout this paper,  for $n\in \mathbb{N}$ and $i\in \Omega_{n+1}$, 
we will denote by $Z_i^n $ the vector field on $\mathcal{P}_n$ determined by
\[
d\iota_n(Z_i^n )=\frac{\partial}{\partial x^i}-\frac{1}{n+1} \sum_{j \in \Omega_{n+1}}\frac{\partial}{\partial x^j}.
\]
For  $X, Y\in \mathfrak{X}(\mathcal{P}_{n})$, we will write
\[
d\iota_n(X)=\sum_{i\in \Omega_{n+1}}X^i\frac{\partial}{\partial x^i}, \quad
d\iota_n(Y)=\sum_{i\in \Omega_{n+1}}Y^i\frac{\partial}{\partial x^i}.
\]
Define  the function $h: \mathbb{R}^{n+1} \to \mathbb{R}$ by 
\[
h\left((x^i)_{i=1}^{n+1}\right) := \sum_{i=1}^{n+1}|x^i|.
\]
We see that  $h\circ \iota_n  \equiv 1$ on $\mathcal{P}_n$, which implies 
\[
\sum_{i\in\Omega_{n+1}} X^i=d\iota_n(X)(h)  = X(h\circ\iota_n) = 0 
\ \quad \text{and} \quad
\sum_{i\in\Omega_{n+1}} Y^i\equiv 0,
\]
consequently
\begin{equation}\label{zero}
\ \  X=\sum_{i\in \Omega_{n+1}} X^i Z_i^{n}=\sum_{i\in \Omega_{n}} X^i (Z_i^{n}-Z_{n+1}^n),
\qquad
Y=\sum_{i\in \Omega_{n+1}} Y^i Z_i^{n}=\sum_{i\in \Omega_{n}} Y^i (Z_i^{n}-Z_{n+1}^n).
\end{equation}
In this expression, the Fisher metric $g_n^F$ at each point corresponds to a diagonal matrix.
We construct a family of invariant $(0,2)$-tensor fields on $\mathcal{P}_{n}$
under scalar patched Markov embeddings, where the corresponding matrix at each point 
is the linear combination of a diagonal matrix and the outer product of a vector with itself.
\begin{definition}
Define a symmetric positive definite $(0,2)$-tensor field $A_n^d $ and 
a symmetric positive semidefinite $(0,2)$-tensor field $A_n^s$ on $\mathcal{P}_{n}$ by 
\begin{align*}
A_n^{d}(X_p, Y_p)
&:=\sum_{i\in \Omega_{n+1}} X_p(\ell^i\circ \iota_n)Y_p(\ell^i\circ \iota_n)
=\sum_{i\in \Omega_{n+1}} \frac{X^i_pY^i_p}{p(i)^2},\\
A_n^{s}(X_p, Y_p)
&= \sum_{i,j\in \Omega_{n+1}} X_p(\ell^i \circ \iota_n) Y_p(\ell^j \circ \iota_n)
=\sum_{i,j\in \Omega_{n+1}}\frac{X^i_p}{p(i)}\frac{Y^j_p}{p(j)},
\end{align*}
for any $p\in \mathcal{P}_{n}$ and $X, Y\in  \mathfrak{X}(\mathcal{P}_{n})$, respectively.
For $\lambda,\mu\in \mathbb{R}$, set $A_n^{\lambda,\mu}:=\lambda A^{d}_n +\mu A^{s}_n$.
\end{definition}
%
\begin{proposition}\label{inv}
The pullback of $A_{n+1}^{\lambda,\mu}$ by each scalar patched Markov embedding from $\mathcal{P}_n$ to $\mathcal{P}_{n+1}$  is $A_n^{\lambda,\mu}$.
Conversely, if the pullback of $A_{n+1}^{\lambda, \mu}$ by  a patched Markov  embedding 
$G_n$ from $\mathcal{P}_{n}$ to $\mathcal{P}_{n+1}$ is  $A_n^{\lambda, \mu}$, 
then either $\lambda, \mu=0$ or $G_n$ is a scalar patched Markov  embedding.
\end{proposition}
\begin{proof}
Fix a  patched Markov embedding  $G_n: \mathcal{P}_{n} \to \mathcal{P}_{n+1}$ of the form
\[
G_n(p):=\sum_{i\in \Omega_{n+1}}p(i)\left(a_i \delta_{\sigma (i)}+(1-a_i)\delta_{\sigma(n+2)}\right)
\quad
\text{for any $p\in \mathcal{P}_{n}$}.
\]
If  $G_n$ is a scalar patched Markov embedding, namely $a_i\equiv \alpha$, then we have 
\[
dG_n (Z_i^n)
=\alpha \left(Z_{\sigma(i)}^{n+1}-\frac{1}{n+1}\sum_{j\in \Omega_{n+1}} Z_{\sigma(j)}^{n+1}\right),\qquad
dG_n (X)
=\alpha\sum_{i\in \Omega_{n+1}}X^i Z_{\sigma(i)}^{n+1},
\]
which implies 
\begin{align*}
A_n^{\lambda, \mu}(X_p,Y_p)
&=
\lambda\sum_{i\in \Omega_{n+1}} \frac{X^i_pY^i_p}{p(i)^2} 
+\mu\sum_{i,j \in \Omega_{n+1}}\frac{X^i_p}{p(i)}\frac{Y^j_p}{p(j)} \\ 
&=
\lambda \sum_{i \in \Omega_{n+1}} \frac{\alpha X^{i}_p \alpha Y^{i}_p }{\{G_n(p)(\sigma(i))\}^2} 
+
\mu\sum_{i,j \in \Omega_{n+1}}\frac{\alpha X^{i}_p}{G_n(p)(\sigma(i))}\frac{\alpha Y^{j}_p}{G_n(p)(\sigma(j))} \\
&=
A_{n+1}^{\lambda, \mu}\left(dG_n(X_p), dG_n(Y_p)\right)
\end{align*}
for any $p\in \mathcal{P}_{n}$.
This proves the first claim.

We next assume that the pullback of $A_{n+1}^{\lambda, \mu}$ by  $G_n$ is  $A_n^{\lambda,\mu}$. 
If  $G_n$ is not a scalar patched Markov  embedding,  then there exist $i,j \in\Omega_{n+1}$ such that $a_i\neq a_j$. 
It turns out that 
\begin{align*}
&d G_n (Z_i^n-Z_j^n)=a_i(Z_{\sigma(i)}^{n+1}-Z_{\sigma(n+2)}^{n+1})-a_j(Z_{\sigma(j)}^{n+1}-Z_{\sigma(n+2)}^{n+1}),\\
&\lambda\left(\frac{1}{p(i)^2}+\frac{1}{p(j)^2}\right)+\mu\left(\frac{1}{p(i)}+\frac{-1}{p(j)}\right)^2 \\
&=A_n^{\lambda,\mu}\left((Z_i^n-Z_j^n)_p, (Z_i^n-Z_j^n)_p\right)
=A_{n+1}^{\lambda,\mu}\left(dG_n\left((Z_i^n-Z_j^n)_p\right),dG_n\left((Z_i^n-Z_j^n)_p\right) \right)\\
&=
\lambda\left[
\frac{a_i^2}{\{ G_n(p)(\sigma(i))\}^2}
+\frac{a_j^2}{ \left\{ G_n(p)(\sigma(j))\right\}^2}
+\frac{(-a_i+a_j)^2}{\left\{ G_n(p)(\sigma(n+2)) \right\}^2}\right] \\
&\quad+\mu\left[ \frac{a_i}{G_n(p)(\sigma(i))}+\frac{-a_j}{G_n(p)(\sigma(j))} +\frac{-a_i+a_j}{G_n(p)(\sigma(n+2))} \right]^2
\qquad\qquad
\text{for any $p\in \mathcal{P}_{n}$.}
\end{align*}
Thus we should have
\begin{align*}
\frac{(\lambda+\mu)(-a_i+a_j)}{ G_n(p)(n+2)}
+2\mu \left(\frac{1}{p(i)}-\frac{1}{p(j)}\right)=0 \quad \text{for any $p\in \mathcal{P}_{n}$},
\end{align*}
which  holds true  if and only if $\lambda+\mu=0$ and $\mu=0$, that is, $\lambda, \mu=0$.
This completes the proof of the second claim.
\end{proof}
\begin{remark}\label{Ca}
We comment on the work of  Campbell \cite{Ca}.
Set 
\[
\mathbb{R}_+^{n+1}:=\{ x=(x^i)_{i=1}^{n+1} \in \mathbb{R}^{n+1} \ |\ x^i>0 \}.
\]
Since $\mathbb{R}_+^{n+1}$ can be regarded as the cone over $\mathcal{P}_n$, Markov embeddings are naturally extended to $\mathbb{R}_+^{n+1}$.
Let  $\widetilde{g}_{n+1}^F$ be  a $C^\infty$- Riemannian metric on $\mathbb{R}_+^{n+1}$.
Campbell  showed that a family  $\{\widetilde{g}_{n+1}^F\}_{n \in \mathbb{N}}$ of Riemannian metrics is  invariant under the extensions of Markov embeddings
if and only if there exist functions $\widetilde{\lambda}, \widetilde{\mu} \in C^\infty(\mathbb{R}^1_+)$ satisfying 
$\widetilde{\lambda}>0$ and $\widetilde{\lambda}+ \widetilde{\mu}>0$ such that 
\[
\widetilde{g}_{n+1}^F\left( \left(\frac{\partial}{\partial x^i}\right)_x, \left(\frac{\partial}{\partial x^j}\right)_x  \right)
=\widetilde{\lambda}\left(h(x) \right)\frac{h(x)}{x^i} \delta_{ij}+ \widetilde{\mu}\left(h(x)\right)
\quad
\text{for any $x\in \mathbb{R}^{n+1}_+$.}\
\]
In this case, the pullback of $\widetilde{g}_{n+1}^F$ by the embedding $\iota_n:\mathcal{P}_n \rightarrow \mathbb{R}_+^{n+1}$ 
coincides with
 $\widetilde{\lambda}(1)g_{n}^F$.
Let $j_n:\mathcal{P}_{n+1}  \to  \mathbb{R}_+^{n+1} $ be the projection  defined  by  $j_n(p):=(p(i))_{i=1}^{n+1}$.
Due to an similarity between $A_{n+1}^{\lambda,\mu}$ and $\widetilde{g}_{n+1}^F$, 
one might expect that the pullback metric of $\widetilde{g}_{n+1}^F$ by $j_n$  is $A_{n+1}^{\lambda,\mu}$.
However this is not true.
\end{remark}
We state common properties of invariant families under scalar patched Markov  embeddings.
In what follows,  given $\sigma\in \mathfrak{S}_{n+2}$ and $\alpha \in (0,1)$,  
let $G_n^{\alpha, \sigma}: \mathcal{P}_{n} \to \mathcal{P}_{n+1}$ be  a scalar patched Markov  embedding defined by
\[
{G}_n^{\alpha,\sigma}(p):=\alpha \sum_{i\in \Omega_{n+1}}p(i)\delta_{\sigma(i)} +(1-\alpha) \delta_{\sigma(n+2)}
\quad
\text{for any $p\in \mathcal{P}_n$.}
\]
In particular, for the identity permutation $\mathrm{id}\in \mathfrak{S}_{n+2}$, we write $G_n^\alpha:=G_n^{\alpha,\mathrm{id}} $.
Denote by  $b_n $ the uniform probability measure on $\Omega_{n+1}$, that is, 
\[
b_n \equiv \frac1{n+1}\quad \text{on}\quad \Omega_{n+1}.
\]
We define $p_{u} \in \mathcal{P}_1$ for ${u} \in (0,1)$  by
\[
p_{u}(1)={u}, \qquad p_{u}(2)=1-{u}.
\]
Then $p_{1/2}=b_1$.
For  $Z\in \mathfrak{X}(\mathcal{P}_1)$,  $Z_{p_u}$ is abbreviated to $Z_u$.
\begin{proposition}\label{invariant}
Let  $A_n$ be  a symmetric $(0,2)$-tensor field  on $\mathcal{P}_{n}$.
Assume that for any $n\in \mathbb{N}$ and  each scalar patched Markov  embedding  $G_n:\mathcal{P}_n \to \mathcal{P}_{n+1}$, 
the pullback of $A_{n+1}$ by  $G_n$ is always $A_n$.
\begin{enumerate}
\item
For $Z\in \mathfrak{X}(\mathcal{P}_1)$ and $u\in(0,1)$,  then  $A_1(Z_u,Z_u)=A_1(Z_{1-u},Z_{1-u})$.
\item
For distinct $i, j \in \Omega_{n+1}$, 
the two quantities 
\[
\frac{A_n\left( (Z_i^{n})_{b_n}, (Z_i^{n})_{b_n}   \right)}{n(n+1)}, \quad
-\frac{ A_n\left( (Z_i^{n})_{b_n}, (Z_j^{n})_{b_n}   \right)}{n+1}
\]
coincide with each other.
The  quantity is independent of $n \in \mathbb{N}$ and $i,j \in \Omega_{n+1}$.
\end{enumerate}
\end{proposition}
\begin{proof}
For $Z\in \mathfrak{X}(\mathcal{P}_1)$ and $\alpha \in (0,1)$, a direct computation provides 
\begin{align*}
A_1(Z_u,Z_u)
&=A_2(dG_1^{\alpha, (1,2)}(Z_u), dG_1^{\alpha, (1,2)}(Z_u)) \\
&=A_2(dG_1^{\alpha}(Z_{1-u}), dG_1^{\alpha}(Z_{1-u}))
=A_1(Z_{1-u},Z_{1-u})),
\end{align*}
where $(1,2)\in \mathfrak{S}_2$ is the transposition.
This proves (1).

To prove (2), we apply a similar argument as in the proof of Theorem \ref{cencov}. 
For example, see \cite{AJLS-book}*{Theorem 2.1}.
Fix distinct  $i,j \in \Omega_{n+1}$ and $\alpha \in (0,1)$.
We find that 
\begin{align*}
A_n\left( (Z_i^n)_{b_n}, (Z_i^n)_{b_n}\right)
&=
A_{n+1}\left( dG_n^{\alpha, (i,j)}\left( (Z_i^n)_{b_n}\right), dG_n^{\alpha, (i,j)}\left( (Z_i^n)_{b_n}\right) \right)\\
&=
A_{n+1}\left( dG_n^{\alpha}\left( (Z_j^n)_{b_n}\right),dG_n^{\alpha}\left(  (Z_j^n)_{b_n} \right)\right)
=
A_n\left( (Z_j^n)_{b_n}, (Z_j^n)_{b_n}\right).
\end{align*}
Thus  $A_n( (Z_i^n)_{b_n}, (Z_i^n)_{b_n})$  is independent of the choice of $i \in \Omega_{n+1}$.
Similarly, we find that  $A_n( (Z_i^n)_{b_n}, (Z_j^n)_{b_n} )$  is independent of the choice of $i,j\in \Omega_{n+1}$.
Moreover, it turns out that 
\begin{align*}
0=A_n\left( (Z_i^n)_{b_n}, 0\right)
&=A_n\left((Z_i^n)_{b_n}, \sum_{k\in \Omega_{n+1}} \left(Z_k^{n}\right)_{b_n} \right)\\
&=
A_n\left( (Z_i^n)_{b_n}, (Z_i^n)_{b_n}\right)+n A_n\left( (Z_i^n)_{b_n}, (Z_j^n)_{b_n}\right).
\end{align*}
If we choose $\alpha=(n+1)/(n+2)$, then  $G_n^{\alpha}(b_n)=b_{n+1}$ and 
\begin{align*}
&A_n\left( (Z_i^n)_{b_n}, (Z_i^n)_{b_n} \right)\\
&=
A_{n+1}\left( dG_n^{\alpha}\left( (Z_i^n)_{b_n}\right),dG_n^{\alpha}\left( (Z_i^n)_{b_n}\right) \right)\\
&=
\alpha^2 
A_{n+1}\left(
(Z_i^{n+1})_{b_{n+1}}-\frac{1}{n+1} \sum_{k\in \Omega_{n+1}} (Z_k^{n+1})_{b_{n+1}},\,
(Z_i^{n+1})_{b_{n+1}}-\frac{1}{n+1} \sum_{k\in \Omega_{n+1}}  (Z_k^{n+1})_{b_{n+1}}
\right)\\
&=
\alpha^2
\left\{\frac{n}{n+1}
\cdot
A_{n+1}\left((Z_i^{n+1})_{b_{n+1}},\, (Z_i^{n+1})_{b_{n+1}}\right)
-\frac{n}{n+1}
\cdot
A_{n+1}\left((Z_i^{n+1})_{b_{n+1}},\, (Z_j^{n+1})_{b_{n+1}}\right)
\right\}\\
&=
\frac{n+1}{n+2}\cdot \frac{n}{n+1}
\cdot
A_{n+1}\left((Z_i^{n+1})_{b_{n+1}},\,(Z_i^{n+1})_{b_{n+1}}\right).
\end{align*}
This completes the proof of (2).
\end{proof}
\begin{proposition}\label{key}
For  $p\in \mathcal{P}_n$ and distinct $i,j\in \Omega_{n+1}$,
there is a composition $ H_p^{ij}:\mathcal{P}_1 \to \mathcal{P}_n$ of scalar patched Markov  embeddings   such that 
\begin{align*}
H_p^{ij}(p_u)=p \quad \text{for\ }u:=\frac{p(i)}{p(i)+p(j)}, \qquad
Z_i^n-Z_j^n=\frac{2dH_p^{ij}(Z_1^1)}{p(i)+p(j)}.
\end{align*}
\end{proposition}
\begin{proof}
It is enough to show the case $(i,j)=(1,2)$.
Given $p\in \mathcal{P}_n$,
define 
$\{\alpha_k\}_{k=1}^{n-1}$ inductively by
\[
\alpha_k:=
\begin{cases}
1-p(n+1) &\text{if\ } k={n-1},\\
1-\dfrac{p(k+2)}{\prod_{l={k+1}}^{n-1} \alpha_l} &\text{if\ } 1\leq k \leq n-2.
\end{cases}
\]
We find  that 
\begin{align*}
\prod_{l=k}^{n-1} \alpha_l &=1- \sum_{l=k+2}^{n+1} p(l) \quad \text{for}\ 1\leq k\leq n-1.
\end{align*}
Then  $H^{12}_p:=G_{n-1}^{\alpha_{n-1}}\circ \cdots \circ G_1^{\alpha}:\mathcal{P}_1 \to \mathcal{P}_n$
is a desired map.
\end{proof}
A similar argument provides  the following.
\begin{corollary}\label{keycor}
For $p\in \mathcal{P}_n$ with $n\geq 2$ and distinct $i,j,k\in \Omega_{n+1}$,
there is a composition $ H_p^{ijk}:\mathcal{P}_2 \to \mathcal{P}_n$
of scalar patched Markov  embeddings   such that 
\begin{align*}
&H_p^{ijk}(q)=p \quad \text{for\ }q\in \mathcal{P}_2\text{\ defined by\ } (q(1), q(2),q(3))=\frac{(p(i),p(j),p(k))}{p(i)+p(j)+p(k)},\\
&
Z_i^n-Z_j^n=\frac{dH_p^{ijk}(Z_1^2-Z_2^2)}{p(i)+p(j)+p(k)},\qquad
Z_i^n-Z_k^n=\frac{dH_p^{ijk}(Z_1^2-Z_3^2)}{p(i)+p(j)+p(k)}.
\end{align*}
\end{corollary}

\section{Main Theorem}
An invariant family  under  Markov embeddings is uniquely determined (Theorem \ref{cencov}),
but there are at least two types of invariant families  under  scalar patched Markov  embeddings
(Proposition~\ref{inv}).
The reason of this difference can be found  in Remark~\ref{mke}.
On one hand, for a Markov embedding $F^N_n:\mathcal{P}_n \to \mathcal{P}_N$ with $n\leq N$, 
a related Fisher--Neyman sufficient statistic $\kappa: \Omega_{N+1}\to\Omega_{n+1}$ always exists.
On the other hand,  for  a composition $G:\mathcal{P}_{n} \to \mathcal{P}_{N}$ of scalar patched Markov  embeddings, 
a related Fisher--Neyman sufficient statistic $\kappa: \Omega_{N+1}\to\Omega_{n+1}$ exists only if $N\geq n+1$.

In this section, we first formulate a conformal  (resp.\ parallel) condition 
in order to classify invariant families under scalar patched Markov embeddings.
Then we show existence and uniqueness of  invariant families under  scalar patched Markov  embeddings
which satisfies either the  conformal condition or the parallel condition.

Let  $A_n$ be  a symmetric $(0,2)$-tensor field  on $\mathcal{P}_{n}$ 
such  that  $\{A_n\}_{n\in \mathbb{N}}$ is invariant  under scalar patched Markov embeddings, 
that is, 
the pullback of $A_{n+1}$ by   each scalar patched Markov  embedding  from $\mathcal{P}_n$ to $\mathcal{P}_{n+1}$ is always $A_n$.
Given $\sigma \in\mathfrak{S}_3$ and $u\in(0,1)$, define a map $\psi^\sigma_u: (0,1) \to \mathcal{P}_2 $ by 
\[
\psi^\sigma_u(\alpha):=G_1^{\alpha, \sigma}(p_u)=
\alpha u \delta_{\sigma(1)}
+\alpha (1-u) \delta_{\sigma(2)}+(1-\alpha)\delta_{\sigma(3)}
\quad
\text{for any $\alpha\in(0,1)$.}
\]
For  $\alpha\in(0,1)$ and distinct $i,j,k \in \Omega_3$, 
define $u_{\alpha}^{ij} \in (0,1)$  by 
\begin{equation}\label{uij}
u_{\alpha}^{ij}:
=\frac{\psi^{\mathrm{id}}_u(\alpha)(i)}{\psi^{\mathrm{id}}_u(\alpha)(i)+\psi^{\mathrm{id}}_u(\alpha)(j)}
=\frac{\psi_u^\sigma(\alpha)(\sigma(i))}{\psi_u^\sigma(\alpha)(\sigma(i))+\psi_u^\sigma(u)(\sigma(j))}.
\end{equation}
Let $E^{k}$ be the  one-dimensional  subbundle of $T\mathcal{P}_2$ spanned by $Z_{i}^2-Z_{j}^2$.
Then Proposition~\ref{key} leads to  
$H^{\sigma(i)\sigma(j)}_{\psi^\sigma_u(\alpha)}(p_{u_{\alpha}^{ij}}) =\psi^\sigma_u(\alpha)$ and 
\begin{align} \label{27}
dH^{\sigma(i)\sigma(j)}_{\psi^\sigma_u(\alpha)}(Z_1^1)
&=\frac{\psi^{\sigma}_u(\alpha)(\sigma(i))+\psi^{\sigma}_u(\alpha)(\sigma(j))}{2}
\left(Z_{\sigma(i)}^2-Z_{\sigma(j)}^2\right) \in E^{\sigma(k)},
\end{align}
hence  the differential map of $H^{\sigma(i)\sigma(j)}_{\psi^\sigma_u(\alpha)}$ at $p_{u_{\alpha}^{ij}}$  gives
an isometry from $(T\mathcal{P}_1,A_1)$ at $p_{u_{\alpha}^{ij}}$
to $(E^{\sigma(k)}, A_2)$ at ${\psi^\sigma_u(\alpha)}$. 

As for a conformal condition, 
we consider the relation between the two pairs $(E^{\sigma(1)}, E^{\sigma(2)} )$ and $(E^{\sigma(1)}, E^{\sigma(3)})$ with respect to $A_2$.
For  any $W_{\sigma(3)}\in E^{\sigma(3)}$, 
there exists  a unique pair $(W_{\sigma(1)},W_{\sigma(2)}) \in E^{\sigma(1)} \times E^{\sigma(2)}$  such that 
$W_{\sigma(3)}=W_{\sigma(1)}+W_{\sigma(2)}$.
In the case  $W_{\sigma(3)}\neq 0$, 
we deduce from Proposition~\ref{key} that
\begin{align*}
&A_2\left( \left(W_{\sigma(3)}\right)_{\psi^{\sigma}_u(\alpha)}, \left(W_{\sigma(3)}\right)_{\psi^{\sigma}_u(\alpha)}\right)\\
&=
A_2\left( \left(W_{\sigma(3)}\right)_{\psi^{\sigma}_u(\alpha)}, \left(W_{\sigma(1)}\right)_{\psi^{\sigma}_u(\alpha)}\right)
+
A_2\left( \left(W_{\sigma(3)}\right)_{\psi^{\sigma}_u(\alpha)}, \left(W_{\sigma(2)}\right)_{\psi^{\sigma}_u(\alpha)}\right)
\end{align*}
is nonzero if and only if $A_1$ is nondegenerate at $p_{u_{\alpha}^{12}}$.
Moreover, for  general $W_{\sigma(3)}$,  
it follows from Proposition \ref{invariant}~(1) that
\begin{equation}\label{rat}
\begin{split}
&A_2\left( \left(W_{\sigma(3)}\right)_{\psi^{\sigma}_u(\alpha)}, \left(W_{\sigma(1)}\right)_{\psi^{\sigma}_u(\alpha)}\right)
+
A_2\left( \left(W_{\sigma(3)}\right)_{\psi^{\sigma}_u(\alpha)}, \left(W_{\sigma(2)}\right)_{\psi^{\sigma}_u(\alpha)}\right)\\
&=A_2\left( \left(W_{\sigma(3)}\right)_{\psi^{\sigma}_{1-u}(\alpha)}, \left(W_{\sigma(1)}\right)_{\psi^{\sigma}_{1-u}(\alpha)}\right)
+
A_2\left( \left(W_{\sigma(3)}\right)_{\psi^{\sigma}_{1-u}(\alpha)}, \left(W_{\sigma(2)}\right)_{\psi^{\sigma}_{1-u}(\alpha)}\right).
\end{split}
\end{equation}
For  $\beta \in (0,1)$
and  the transposition $(\sigma(i), \sigma(j)) \in \mathfrak{S}_3$, 
since we have 
\[
G_2^{\beta}(\psi^{\sigma}_u(\alpha) )
=G_2^{\beta, (\sigma(1),\sigma(2)   )}(\psi^{\sigma}_{1-u}(\alpha) ),
\]
we observe from the invariance $\{A_n\}_{n\in \mathbb{N}}$ that
\begin{equation}\label{sym}
\begin{split}
&A_2\left( \left(W_{\sigma(3)}\right)_{\psi^{\sigma}_u(\alpha)}, \left(W_{\sigma(1)}\right)_{\psi^{\sigma}_u(\alpha)}\right)\\
&=
A_3\left( dG_2^{\beta}\left( \left(W_{\sigma(3)}\right)_{\psi^{\sigma}_u(\alpha)}\right), dG_2^{\beta}\left(\left(W_{\sigma(1)}\right)_{\psi^{\sigma}_u(\alpha)}\right)\right) \\
&=
A_3\left(
dG_2^{\beta, (\sigma(1),\sigma(2) )}\left( \left(W_{\sigma(3)}\right)_{\psi^{\sigma}_{1-u}(\alpha)} \right), 
dG_2^{\beta, (\sigma(1),\sigma(2) )}\left( \left(W_{\sigma(2)}\right)_{\psi^{\sigma}_{1-u}(\alpha)}\right)\right)\\
&=
A_2\left( \left(W_{\sigma(3)}\right)_{\psi^{\sigma}_{1-u}(\alpha)}, \left(W_{\sigma(2)}\right)_{\psi^{\sigma}_{1-u}(\alpha)}\right).
\end{split}
\end{equation}
For $W_{\sigma(3)}\in E^{\sigma(3)}$ with $W_{\sigma(3)}\neq 0$, 
 \eqref{rat} and \eqref{sym} provide that the ratio $r(  \psi^{\sigma}_u(\alpha) )$ between
\[
A_2\left( \left(W_{\sigma(3)}\right)_{\psi^{\sigma}_u(\alpha)}, \left(W_{\sigma(1)}\right)_{\psi^{\sigma}_u(\alpha)}\right) 
\qquad\text{and }\qquad
A_2\left( \left(W_{\sigma(3)}\right)_{\psi^{\sigma}_{u}(\alpha)}, \left(W_{\sigma(2)}\right)_{\psi^{\sigma}_{u}(\alpha)}\right)
\]
satisfies 
\[
r\left(  \psi^{\sigma}_u(\alpha) \right) \cdot r\left(  \psi^{\sigma}_{1-u}(\alpha) \right)=1
\qquad
\left(\text{resp.\ } r(  \psi^{\sigma}_{1/2}(\alpha) )\equiv 1\right)
\]
if $A_1$ at  $p_{u_{\alpha}^{12}}$  (resp. at $b_1$) is nondegenerate, where we put $\infty\cdot 0:=1$ by convention.
We assume that the ratio $r(  \psi^{\sigma}_u(\alpha) )$   depends not $\sigma\in \mathfrak{S}_3$ and $\alpha\in (0,1)$, 
but only on  the ratio between $u$ and $1-u$.
\begin{enumerate}
\setlength{\leftskip}{-7pt}
\item[{\bf (C1)}]
There exists $r:(0,\infty)\to (0,\infty)$ such that $r(t)r(1/t)=1$  and 
\[
A_2\left( \left(W_{\sigma(3)}\right)_{\psi^{\sigma}_u(\alpha)}, \left(W_{\sigma(1)}\right)_{\psi^{\sigma}_u(\alpha)}\right)
=r\left(\frac{u}{1-u}\right)\cdot
A_2\left( \left(W_{\sigma(3)}\right)_{\psi^{\sigma}_{u}(\alpha)}, \left(W_{\sigma(2)}\right)_{\psi^{\sigma}_{u}(\alpha)}\right)
\]
for $\alpha,u \in (0,1), \sigma \in \mathfrak{S}_3$ and $W_{\sigma(3)}\in E^{\sigma(3)}$ 
decomposed as 
$W_{\sigma(3)}=W_{\sigma(1)}+W_{\sigma(2)}$,
where 
$W_\sigma(1)\in E^{\sigma(1)}$ and $W_\sigma(2)\in E^{\sigma(2)}$.
\end{enumerate}
Note that $A_2^d$ satisfies {\bf (C1)} with $r(t)=t^2$.

We next discuss a parallel condition.
If $A_1$ is positive definite, 
then $Z:=Z_1^1/\sqrt{A_1(Z_1^1,Z_1^1)}$ is parallel along 
the curve $\{p_u\}_{u\in (0,1)}$ with respect to $A_1$.
Due to the  invariance of $\{A_n\}_{n\in \mathbb{N}}$, 
$dG_1^{\alpha, \sigma}(Z)$  is a parallel vector filed along the curve 
$\{G_1^{\alpha, \sigma}(p_u)= \psi^\sigma_u(\alpha) \}_{u\in (0,1)}$ with respect to~$A_2$.
We observe  from \eqref{27} that  $E^{\sigma(k)}$ consists of parallel vector fields.
By the positive definiteness of $A_1$, 
$(E^{\sigma(i)}_{\psi^\sigma_u(\alpha)}, A_2)$ and  $(E^{\sigma(j)}_{\psi^\sigma_u(\alpha)}, A_2)$ are isometric to each other.
Based on this observation, it is tempting to assume that  $E^{\sigma(i)}$ and  $E^{\sigma(j)}$
are parallel at each point in~$\mathcal{P}_2$.
However,  if $\psi^\sigma_u(\alpha)=b_2$,
then  $E^{\sigma(i)}_{\psi^\sigma_u(\alpha)}$ and  $E^{\sigma(j)}_{\psi^\sigma_u(\alpha)}$ are not parallel
since  equality in the Cauchy--Schwarz inequality of the form
\begin{align*}
&A_2\left( 
dH^{\sigma(i)\sigma(j)}_{\psi^\sigma_u(\alpha)} \left(\left(Z_1^1\right)_{u^{ij}_\alpha }\right),\,
dH^{\sigma(k)\sigma(l)}_{\psi^\sigma_u(\alpha)} \left(\left(Z_1^1\right)_{u^{kl}_\alpha }\right)
\right)^2\\
&\leq 
A_2\left( 
dH^{\sigma(i)\sigma(j)}_{\psi^\sigma_u(\alpha)} \left(\left(Z_1^1\right)_{u^{ij}_\alpha }\right),\,
dH^{\sigma(i)\sigma(j)}_{\psi^\sigma_u(\alpha)} \left(\left(Z_1^1\right)_{u^{ij}_\alpha }\right)
\right)\\
&\qquad\times
A_2\left( 
dH^{\sigma(k)\sigma(l)}_{\psi^\sigma_u(\alpha)} \left(\left(Z_1^1\right)_{u^{kl}_\alpha }\right),\,
dH^{\sigma(k)\sigma(l)}_{\psi^\sigma_u(\alpha)} \left(\left(Z_1^1\right)_{u^{kl}_\alpha }\right)
\right)\\
&=A_1\left(  \left(Z_1^1\right)_{u^{ij}_\alpha }, \left(Z_1^1\right)_{u^{ij}_\alpha }\right) 
\times
A_1\left(\left(Z_1^1\right)_{u^{kl}_\alpha }, \left(Z_1^1\right)_{u^{kl}_\alpha }\right)
\end{align*}
does not hold for two pairs of distinct points $(i,j),(k,l)\in \Omega_3$  with $\{i,j\}\neq \{k,l\}$ by Proposition~\ref{invariant}~(2).
To have a parallel condition, 
we assume that $A_1$ is positive semidefinite and  degenerate at $b_1$.
Then  the behavior of $A_1( (Z_1^1)_u,(Z_1^1)_u)$ may change $u=1/2$ .
We regard this change  as the change of the sign of $2u-1$,
where we put  $\mathrm{sgn}(0):=0$ by convention.
\begin{enumerate}
\setlength{\leftskip}{-7pt}
\item[{\bf (C2)}]
In addition to being positive semidefinite, $A_1$  is degenerate at $b_1$.  
Given two pairs of distinct points $(i, j), (k,l) \in \Omega_3$, 
\begin{align*}
&A_2\left( 
dH^{\sigma(i)\sigma(j)}_{\psi^\sigma_u(\alpha)} \left(\left(Z_1^1\right)_{u^{ij}_\alpha }\right),\,
dH^{\sigma(k)\sigma(l)}_{\psi^\sigma_u(\alpha)} \left(\left(Z_1^1\right)_{u^{kl}_\alpha }\right)
\right)\\
&=
\mathrm{sgn}\left(2u_\alpha^{ij}-1\right)
\sqrt{
A_1\left(  \left(Z_1^1\right)_{u^{ij}_\alpha }, \left(Z_1^1\right)_{u^{ij}_\alpha }\right) }
\times
\mathrm{sgn}\left(2u_\alpha^{kl}-1\right)
\sqrt{A_1\left(\left(Z_1^1\right)_{u^{kl}_\alpha }, \left(Z_1^1\right)_{u^{kl}_\alpha }\right)}
\end{align*}
holds for  $\alpha,u \in (0,1)$ and  $\sigma \in \mathfrak{S}_3$,
where $u^{ij}_\alpha$ is defined  by \eqref{uij}.
\end{enumerate}
It is easy to check that   $A_1^s$ and $A_2^s$ satisfy {\bf (C2)}.

We are now ready to state the main theorem of this paper.
\begin{theorem}\label{diag}
For $n\in \mathbb{N}$, let $A_n$ be a symmetric  $(0,2)$-tensor field on $\mathcal{P}_{n}$.
Assume that for any $n\in \mathbb{N}$ and  each scalar patched Markov  embedding  $G_n:\mathcal{P}_n \to \mathcal{P}_{n+1}$, 
the pullback of $A_{n+1}$ by  $G_n$ is always $A_n$.
\begin{enumerate}
\setlength{\leftskip}{-7pt}
\item
If $A_2$ satisfies {\bf (C1)},
then  there exists $\lambda \in \mathbb{R}$  independent of $n\in \mathbb{N}$  such that  $A_n=\lambda A_n^{d}$.
\item
If $A_1$ and $A_2$ satisfy {\bf (C2)},
then  there exists $\mu \geq 0$  independent of $n\in \mathbb{N}$  such that  $A_n=\mu A_n^{s}$.
\end{enumerate}
\end{theorem}
\begin{proof}
In this proof, set $Z:=Z_1^1 \in \mathfrak{X}(\mathcal{P}_1)$. 
Fix 
\[
W_1:=Z_3^2-Z_2^2\in E^{1}, \qquad
W_2:=Z_1^2-Z_3^2\in E^{2},\qquad
W_3:=Z_1^2-Z_2^2\in E^{3}.
\]
Then we have  $W_3=W_1+W_2$ on  $\mathcal{P}_2$.
Taking $p\in \mathcal{P}_n$ and  distinct $i,j \in \Omega_n$, 
we define $u^i\in  (0,1)$ and $q^{ij} \in \mathcal{P}_2$   by
\[
u^i=\frac{p(n+1)}{p(n+1)+p(i)},\qquad
\left(
q^{ij}(1),
q^{ij}(2),
q^{ij}(3)
\right)
=\frac{\left(p(n+1), p(i),p(j) \right)    }{p(n+1)+p(i)+p(j)}, 
\]
respectively.
It follows from Proposition~\ref{key} and Corollary \ref{keycor} that
\begin{equation}\label{Z}
\begin{split}
(Z_{n+1}^n-Z_i^n)_p
&
=\frac{2d H_p^{n+1i} \left( Z_{u^i}\right)}{p(n+1)+p(i)}
=\frac{d H_p^{n+1ij} \left( \left(W_3\right)_{q^{ij}}\right)}{p(n+1)+p(i)+p(j)}
=\frac{d H_p^{n+1ij} \left( 2dH^{12}_{q^{ij}}(Z_{u^i})\right)}{p(n+1)+p(i)},\\
(Z_{n+1}^n-Z_j^n)_p
&
=\frac{2d H_p^{n+1j} \left( Z_{u^j}\right)}{p(n+1)+p(i)}
=\frac{d H_p^{n+1ij} \left( \left(W_2\right)_{q^{ij}}\right)}{p(n+1)+p(i)+p(j)}
=\frac{d H_p^{n+1ij} \left( 2dH^{13}_{q^{ij}}(Z_{u^j}) \right)}{p(n+1)+p(j)}.
\end{split}
\end{equation}
Note that applying  \eqref{zero} yields 
\begin{align}\label{A}
A_n(X_p, Y_p)
&=\sum_{i,j\in \Omega_n} X^i_pY^j_p A_n\left( (Z_{n+1}^n-Z_i^n)_p, (Z_{n+1}^n-Z_j^n)_p\right)\quad
\text{for any\ }X,Y\in\mathfrak{X}(\mathcal{P}_n).
\end{align}

Assume {\bf (C1)}.
For $\alpha, u \in (0,1)$, set $p^\alpha_u:=\psi^{\mathrm{id}}_u(\alpha) \in \mathcal{P}_2$.
We  observe from Proposition~\ref{key}  that  
\begin{align}\label{ratio}
4u^2 A_1\left( Z_{u},Z_{u}\right)\notag
&=
4\left(\frac{p_u^\alpha(1)}{p_u^\alpha(1)+p_u^\alpha(2)}\right)^2
A_1\left( Z_{u},Z_{u}\right)\\ \notag
&=
\frac{4p_u^\alpha(1)^2}{(p_u^\alpha(1)+p_u^\alpha(2))^2}
A_2 \left(dH^{12}_{p_u^\alpha}\left(Z_{u}\right),
dH^{12}_{p_u^\alpha}\left(Z_{u}\right) \right)\\ 
&=
p_u^\alpha(1)^2
A_2\left( (W_3)_{p^\alpha_u},  (W_3)_{p^\alpha_u}\right)\\ \notag
&=\left\{ r\left( \frac{u}{1-u}  \right)+1\right\}
p_u^\alpha(1)^2 A_2\left( (W_3)_{p^\alpha_u},  (W_2)_{p^\alpha_u}\right),
\end{align}
providing that $p^\alpha_u(1)^2A_2\left((W_3)_{p^\alpha_u},(W_2)_{p^\alpha_u}\right)$  depends 
not on $\alpha \in (0,1)$ but only on  the ratio between $u$ and $1-u$.
Similarly, for
\[
v(\alpha,u):=\frac{p_u^\alpha(1)}{p_u^\alpha(1)+p_u^\alpha(3)}\in(0,1),
\]
we have $p_u^\alpha=\psi_u^{\mathrm{id}}(\alpha) =  \psi^{(2,3)}_{v(\alpha,u)}(p_u^\alpha(1)+p_u^\alpha(3))$ and 
\begin{align*}
4v(\alpha, u)^2
A_1\left( Z_{v(\alpha,u)}, Z_{v(\alpha,u)}\right)
&=
p_u^\alpha(1)^2
A_2\left( (W_3)_{p^\alpha_u},  (W_3)_{p^\alpha_u}\right)\\
&=\left\{ r\left( \frac{v(\alpha,u)}{1-v(\alpha,u)}  \right)+1\right\}
p_u^\alpha(1)^2 A_2\left( (W_3)_{p^\alpha_u},  (W_2)_{p^\alpha_u}\right).
\end{align*}
Setting $\alpha(u):=1/(1+u)$, 
we have   $v(\alpha(u),u)=1/2$ and 
\begin{align*}
 A_1\left( Z_{1/2},Z_{1/2}\right)
&=2 p_u^{\alpha(u)}(1)^2 A_2\left( (W_3)_{p^{\alpha(u)}_u},  (W_2)_{p^{\alpha(u)}_u}\right),
\end{align*}
where  we used  $r(1)=1$. 
Since 
$p^\alpha_u(1)^2A_2\left((W_3)_{p^\alpha_u},(W_2)_{p^\alpha_u}\right)$  is independent of $\alpha \in (0,1)$,
if we set   
\[
 \lambda:=\frac12 A_1( Z_{1/2}, Z_{1/2}),
\]
then  we obtain
\begin{align*}
 \lambda
=p_u^\alpha(1)^2 A_2\left( (W_3)_{p^\alpha_u},  (W_2)_{p^\alpha_u}\right)
\qquad\qquad
\text{for any\ }\alpha\in(0,1).
\end{align*}
Combining this with the property that $\{ p^\alpha_u\}_{\alpha, u\in(0,1)}=\mathcal{P}_2$ and \eqref{sym} lead to
\[
A_2\left( (W_3)_{q},  (W_2)_{q}\right)=\frac{\lambda}{(q(1))^2}, 
\quad
A_2\left( (W_3)_{q},  (W_1)_{q}\right)=\frac{\lambda}{(q(2))^2}
\qquad
\text{for any\ }q\in \mathcal{P}_2.
\]
By \eqref{ratio}, it turns out that
\begin{align*}
 A_1(Z_u, Z_u)
&=\left(\frac{p_u^\alpha(1)}{4u^2}\right)^2 A_2\left( (W_3)_{p^\alpha_u},  (W_3)_{p^\alpha_u}\right) \\
&=\frac{\alpha^2}{4}A_2\left( (W_3)_{p^\alpha_u}, (W_1)_{p^\alpha_u}+ (W_2)_{p^\alpha_u}\right)
=
\frac{\lambda}{4}\left(\frac1{u^2}+\frac1{(1-u)^2}\right).
\end{align*}
Then we apply \eqref{A} together with \eqref{Z} for $X,Y \in\mathfrak{X}(\mathcal{P}_n) $
to have 
\begin{align*}
A_n(X_p, Y_p)
&=\sum_{i\in \Omega_n}X^i_pY^i_p
A_n
\left(
\frac{d H_p^{n+1i} \left(Z_{u^i}\right)}{p(n+1)+p(i)},
\frac{d H_p^{n+1i} \left(Z_{u^i}\right)}{p(n+1)+p(i)} \right)\\
&\quad +
\sum_{\substack{i,j\in \Omega_n, \\ i\neq j}}X^i_pY^j_p
A_n\left(
\frac{dH_p^{n+1\,ij}\left( \left(W_3\right)_{q^{ij}}\right)}{p(n+1)+p(i)+p(j)},
\frac{dH_p^{n+1\,ij}\left( \left(W_2\right)_{q^{ij}}\right)}{p(n+1)+p(i)+p(j)} \right)\\
&=\sum_{i\in \Omega_n}\frac{4X^i_pY^i_p}{(p(i)+p(n+1))^2}
\cdot
\frac{\lambda}{4}\left(\frac{1}{(u^i)^2}+\frac{1}{(1-u^i)^2}\right)\\
&\quad+
\sum_{\substack{i,j\in \Omega_n, \\ i\neq j}}\frac{X^i_p Y^j_p}{(p(n+1)+p(i)+p(j))^2} \cdot
\frac{\lambda}{q^{ij}(1)^2}\\
&=
\lambda \sum_{i\in \Omega_n}\frac{X^i_p}{p(n+1)}\frac{Y^i_p }{p(n+1)}
+
\lambda \sum_{i\in \Omega_n}\frac{X^i_p}{p(i)}\frac{Y^i_p }{p(i)}
+
\lambda\sum_{\substack{i,j\in \Omega_n, \\ i\neq j}} \frac{X^i_p}{p(n+1)}\frac{Y^j_p}{p(n+1)}\\
&=\lambda \sum_{i\in \Omega_{n+1}}\frac{X^i_p}{p(i)}\frac{Y^i_p }{p(i)},
\end{align*}
which proves (1).


To prove (2),  assume {\bf (C2)} and set 
\[
M(u):=\mathrm{sgn}(2u-1) \sqrt{A_1( Z_u, Z_u) }
\quad
\text{for any\ }u\in(0,1).
\]
By Proposition~\ref{invariant}~(1), we see that $M(u)=-M(1-u)$.
In particular, $M(1/2)=0$.
For $q\in \mathcal{P}_{2}$, 
set
\begin{align*}
u:=\frac{q(1)}{q(1)+q(2)}, \quad v:=\frac{q(1)}{q(1)+q(3)},  \quad w:=\frac{q(3)}{q(3)+q(2)}.
\end{align*}
Then we observe from Proposition~\ref{key} that
\[
\frac{dH_q^{12}(Z_u)}{q(1)+q(2)}
=\frac12\left( W_3  \right)_q
=
\frac12\left(  W_1 \right)_q
+
\frac12\left(  W_2 \right)_q
=
\frac{dH_q^{32}( Z_w)}{q(3)+q(2)}
+
\frac{dH_q^{13}( Z_{v})}{q(1)+q(3)},
\]
which implies 
\begin{align*}
\left( \frac{M(u)}{q(1)+q(2)}\right)^2
&=\frac{A_{1}( Z_u, Z_u)}{\{ q(1)+q(2)\}^2}
=A_2 \left(\frac{dH_q^{12}( Z_u)}{q(1)+q(2)}, \frac{dH_q^{12}( Z_u)}{q(1)+q(2)}\right)\\
&=
A_2 \left(\frac{dH_q^{32}( Z_{w})}{q(3)+q(2)}, \frac{dH_q^{32}( Z_{w})}{q(3)+q(2)}\right)
+
A_2 \left(\frac{dH_{\psi^\sigma_u(\alpha)}^{13}( Z_{v})}{q(1)+q(3)}, \frac{dH_q^{13}( Z_{v})}{q(1)+q(3)}\right)\\
&\qquad+
2A_2 \left(\frac{dH_q^{32}( Z_w)}{q(3)+q(2)}, \frac{dH_q^{13}( Z_v)}{q(1)+q(3)}\right).
\end{align*}
We apply {\bf (C2)} for $\sigma=\mathrm{id}\in \mathfrak{S}_3$ and $\alpha=q(1)+q(2)\in(0,1)$, namely
\[
q=\psi_u^{\mathrm{id}}(\alpha),\qquad
u^{13}_{\alpha}=\frac{q(1)}{q(1)+q(3)}=v, \qquad
u^{32}_{\alpha}=\frac{q(3)}{q(3)+q(2)}=w,
\]
to have
\begin{align*}
A_2 \left(\frac{dH_q^{32}( Z_w)}{q(3)+q(2)},\frac{dH_q^{13}( Z_v)}{q(1)+q(3)}\right)
&=\frac{M(w)}{q(3)+q(2)}\cdot \frac{ M(v)}{q(3)+q(1)}.
\end{align*}
These  yield 
\begin{align*}
\left( \frac{M(u)}{q(1)+q(2)}\right)^2
&=\left(\frac{M(w)}{q(3)+q(2)}\right)^2+
\left(\frac{M(v)}{q(3)+q(1)}\right)^2+
2\frac{M(w)}{q(3)+q(2)}\cdot\frac{ M(v)}{q(3)+q(1)}\\
&=
\left(
\frac{M(w)}{q(2)+q(3)}
+
 \frac{M(v)}{q(1)+q(3)}
\right)^2.
\end{align*}
It follows from the relation
\[
\{(1-v)(1-w)+w\}\cdot q(j)
=\begin{cases} {vw} & j=1,\\  (1-v)(1-w) &j=2, \\ {(1-v)w} & j=3,
\end{cases}
\]
together with the choice  $v,w \leq 1/2$ or $v,w\geq 1/2$ that 
\begin{align*}
M(u)
&= M\left( \frac{vw}{vw+(1-v)(1-w)}\right)\\
&=
\frac{vw+(1-v)(1-w)}{1-v}M(w)+\frac{vw+(1-v)(1-w)}{w}M(v).
\end{align*}
Replacing  $v$ and $w$ in the above equation, we find that
\[
\left(\frac{1}{w}-\frac{1}{1-w}\right) M(v)
=
\left(\frac{1}{v}-\frac{1}{1-v}\right) M(w).
\]
This with the relation $M(u)=-M(1-u)$ implies 
\[
M(u)^2
= 
\frac{ M(\frac13)^2 \left(\frac{1}{u}-\frac{1}{1-u}\right)^2}{\left(\frac{1}{\frac13}-\frac{1}{1-\frac13}\right)^{2}}
=\frac{4}{9}M\left(\frac13\right)^2\left(\frac{1}{u}-\frac{1}{1-u}\right)^2
\quad
\text{for any\ }u\in (0,1).
\]
Thus we have 
\[
M(u)= \frac{\sqrt{\mu}}{2} \cdot \left(\frac{1}{1-u}-\frac1u\right)
\quad \text{for\ }u\in (0,1),
\quad \text{where}\ \mu:=\frac{16}{9}M\left(\frac13\right)^2\geq 0.
\]
We apply \eqref{A} and  \eqref{Z} together with {\bf (C2)} for $X,Y \in\mathfrak{X}(\mathcal{P}_n)$ to have
\begin{align*}
A_n(X_p, Y_p)
&=
\sum_{i,j\in \Omega_{n}}
X^i_p Y^j_p
A_2\left(\frac{2dH^{12}_{q^{ij}}(Z_{u^i})}{p(n+1)+p(i)},
\frac{2dH^{13}_{q^{ij}}(Z_{u^j})}{p(n+1)+p(j)}\right)\\
&=
\sum_{i,j\in \Omega_{n}}
X^i_p Y^j_p
\frac{2M(u^i)}{p(n+1)+p(i)}
\frac{2M(u^j)}{p(n+1)p(j)}\\
&=
\mu\sum_{i,j\in \Omega_{n}}
X^i_p Y^j_p
\left( \frac{1}{p(n+1)}-\frac{1}{p(i)}\right)
\left( \frac{1}{p(n+1)}-\frac{1}{p(j)}\right)\\
&=\mu \sum_{i,j\in \Omega_{n+1}}\frac{X^i_p}{p(i)}\frac{Y^j_p}{p(j)},
\end{align*}
where we used the fact $\sum_{i\in \Omega_{n+1}} X_p^i=0$ and $\sum_{i\in \Omega_{n+1}} Y_p^i=0$
in the last equality.
This completes the proof of (2).
\end{proof}
\begin{remark} 
In  Theorems~\ref{cencov}, 
we  impose the continuity of $\{A_n\}_{n\in \mathbb{N}}$ in addition to the invariance under Markov embeddings.
In Theorem~\ref{diag}, 
we do not need  the continuity of $\{A_n\}_{n\in \mathbb{N}}$, and the proof is  essentially different from that of  Theorem~\ref{cencov}.
The behavior of $A_1$ and  $A_2$ is essential 
to define  a family $\{A_n\}_{n\in \mathbb{N}}$ of invariant $(0,2)$-tensor fields on $\mathcal{P}_{n}$ under scalar patched Markov embeddings.
\end{remark}
\begin{ack}
{\small
The both authors express their sincere gratitude to Professor Akio Fujiwara for his valuable comments on sufficient statistics.
The authors also  would like to thank Professor Hitoshi Furuhata for stimulating discussion. 
HM was supported in part by JSPS Grant-in-Aid for Scientific Research (KAKENH) 19K03489.
AT was supported in part by KAKENHI 19K03494, 19H01786.
}
\end{ack}

\end{document}